\DeclareSymbolFont{AMSb}{U}{msb}{m}{n}
\documentclass[12pt,a4paper,twoside,leqno,noamsfonts]{amsart}
\usepackage[utopia]{mathdesign}
\usepackage[utf8]{inputenc}
\usepackage[english]{babel}
\usepackage{amsmath, amsthm, amstext}
\usepackage{microtype}

\usepackage{indentfirst}
\usepackage[colorlinks,bookmarks]{hyperref}
\usepackage{braket}
\usepackage{caption}
\usepackage{stmaryrd}
\usepackage{comment}
\usepackage{multirow}
\usepackage{booktabs}
\usepackage[all]{xy}
\usepackage{graphicx}

\usepackage{listings}

\newtheorem{cor}{Corollary}[section]
\newtheorem*{cor*}{Corollary}
\newtheorem{lem}[cor]{Lemma}
\newtheorem*{lem*}{Lemma}
\newtheorem{thm}[cor]{Theorem}
\newtheorem*{thm*}{Theorem}

\newtheorem*{conj*}{Conjecture}
\newtheorem{prop}[cor]{Proposition}
\newtheorem*{prop*}{Proposition}

\theoremstyle{definition}
\newtheorem{defn}[cor]{Definition}

\newtheorem{rmk}[cor]{Remark}

\newcommand{\bC}{\mathbb{C}}

\newcommand{\bP}{\mathbb{P}}

\newcommand{\cB}{\mathcal{B}}

\newcommand{\cV}{\mathcal{V}}
\newcommand{\cX}{\mathcal{X}}

\newcommand{\sF}{\mathscr{F}}

\newcommand{\aut}{\operatorname{Aut}}

\title{Monodromy of general hypersurfaces}

\author{Maria Gioia Cifani}
\address[M.G.C.]{Department of Mathematics 'F. Casorati', University of Pavia, via Ferrata 5, 27100 Pavia, Italy}
\email{mariagioia.cifani@unipv.it}

\begin{document}
\begin{abstract}
Let $X$ be a general complex projective hypersurface in $\bP^{n+1}$ of degree $d>1$. A point $P$ not in $X$ is called uniform if the monodromy group of the projection of $X$ from $P$ is isomorphic to the symmetric group. We prove that all the points in $\bP^{n+1}$ are uniform for $X$, generalizing a result of Cukierman on general plane curves. 
\end{abstract}
\maketitle

\section{Introduction}
The monodromy group of linear projections of irreducible complex projective varieties has been intensively studied. Fixed an irreducible and reduced projective hypersurface $X \subset \bP^{n+1}$, consider its linear projections from a point $P \in \bP^{n+1}$. We want to look at those maps from a topological point of view: in particular, we aim to classify the centres of projection through their monodromy group. We recall that we can give also an algebraic description: indeed, the monodromy group is isomorphic to the Galois group for finite dominant morphisms between irreducible complex varieties \cite[Section I]{H}.

We will say that a point $P$ is \emph{uniform} for $X$ if the monodromy group of the projection from $P$ is the symmetric group, \emph{non uniform} otherwise.
 
A direct consequence of the Castelnuovo's uniform position principle, in the formulation of Harris \cite{JHCurves}, is that a general projection has always symmetric monodromy group. In 2005 Pirola and Schlesinger \cite{PS} improved this result showing that an irreducible and reduced plane curve admits at most a finite number of non uniform points. Moreover, in \cite{CMS} it is proved that smooth surfaces in $\bP^3$ admit at most a finite number of non uniform points. More recently the author, Cuzzucoli and Moschetti \cite{CCM} studied the case of hypersurfaces of higher dimension proving that, except for special configurations, the non uniform locus is contained in linear subspaces of codimension two. In particular, we proved that smooth hypersurfaces admit at most a finite number of non uniform points \cite[Theorem 1.3]{CCM}. 

Examples of smooth hypersurfaces admitting at least a non uniform point are known (see for instance \cite{Miura1} \cite{MY1} for plane curves or \cite{Yoshiara} for hypersurfaces). 
One may ask if every smooth hypersurface admit non uniform points, but the answer is negative. 

In 1999 Fernando Cukierman (\cite{Cuk}) proved that for general plane curves, all the outer points are uniform. In this work we generalize this result proving the following 
\begin{thm}
Let $X \subset \bP^{n+1}$ be a general hypersurface of degree $d>1$. Then all the points $P \in \bP^{n+1}$ are uniform.
\end{thm}
The result was already known for a special class of non uniform points that are the Galois points (\cite[Theorem 1]{Yoshiara}). We remark also that the Theorem extends the result of Cukiermann to inner points of general plane curves. 

The proof combines inner and outer projections and it is based on an induction argument on the degree of the variety: we degenerate the hypersurface $X$ to a limit one given by a general hypersurface $Y$ of degree $d-1$ and a hyperplane. 
The base case of the induction ($d=3$, Theorem \ref{d3}) is consequence of a result of Matsumura and Monsky \cite{MM} saying that a general hypersurface has trivial automorphism group. More in general, the induction step is based on the study of the behaviour of the monodromy group of the projection $\pi_P$ of $X$ from $P$ under degenerations (Lemma \ref{induction}). 
\begin{lem}
Let $P \in \bP^{n+1}$ and let $\pi_0$ be the map $\pi_P$ restricted to $Y$. Then, the monodromy group $M(\pi_0)$ is contained in the monodromy group $M(\pi_P)$.
\end{lem}
This Lemma is based on some classical topological results on homotopy of fibrations (Proposition \ref{lem:Noriprelim}), reported in Section \ref{sec3}. In particular, we considered the case of a family of dominant maps $F: \cX \to Y \times \bP^1$ parametrised by $\bP^1$, where $\cX$ is a flat family of projective varieties of dimension $m$ in $\bP^N$ and $Y$ is a smooth projective variety. For a general $s \in \bP^1$, the fibre over $\bP^1$ is a smooth projective variety $X_s$ of dimension $m$ in $\bP^N$, together with a finite dominant morphism $f_s: X_s \to Y$ of degree $d$. We deduce a result on monodromy groups (Proposition \ref{lemmagenerale}):
\begin{prop}
If $X_s$ is reduced for every $s \in \bP^1$, then  $$M(F)\cong M(f_s).$$
\end{prop}
More generally, if there is a non reduced fibre $X_0$ with a reduced component $Z$, we have (Proposition \ref{redcomponent})
\begin{prop}
The monodromy group of $f_0$ restricted to $Z$ is contained in the monodromy group $M(f_s)$ for a general $s$ in a neighbourhood of $0$. 
\end{prop}

To conclude the proof of the main Theorem, we use results on multiply transitive permutation groups (see Section \ref{permutazioni}). 

 \medskip

\textbf{Notations.}
All the varieties are assumed to be complex and projective. Let $\sF$ be a family of objects parametrised by a scheme $V$. We say the general element of $\sF$ satisfies a certain property if this property holds for every element in a Zariski dense open subset of $V$. Moreover, we will always use the Zariski topology, unless stated otherwise. 

\section{Preliminaries} \label{sec:preliminaries}
\subsection{Monodromy and Galois group}
Let $f:X \to Y$ be a finite dominant morphism of degree $d$ between complex irreducible reduced varieties of the same dimension. Let $U \subset Y$ be a Zariski open set over which $f$ is \'etale, and let $y$ denote a point in $U$. We have a well defined map
$$\mu: \pi_1(U,y) \to \aut\big(f^{-1}(y)\big) \simeq S_d.$$
The image $M(f):=\mu\left(\pi_1(U,y)\right)$ is called \emph{monodromy group} of the map $f$; it is a transitive subgroup of the symmetric group.

We can also describe this group by means of Galois extensions: let $K$ be the Galois closure of the extension $\bC(X)/\bC(Y)$, where $\bC(X),\bC(Y)$ define the fields of rational functions of $X$ and $Y$ respectively. Define the \emph{Galois group} $G(f)$ of the map $f$ to be the Galois group of the field extension $K/\bC(Y)$.
It turns out that $G(f)$ is isomorphic to $M(f)$, see \cite[Section I]{H}. We recall also that Galois group of a field extension $K/\bC(Y)$ is defined as the group of automorphisms of $K$ fixing $\bC(Y)$. 

\subsection{Automorphisms of general hypersurfaces}
Let $V$ be a projective variety; we denote by $\aut(V)$ the group of automorphisms of $V$. We will use the following result of Matsumura and Monsky \cite{MM}:
\begin{thm}\label{autom}
Let $X$ be a general hypersurface in $\bP^{n+1}$, with $ n\geq 2$ and $d\geq 3$. Then $\operatorname{Aut}(X)$ is trivial.
\end{thm}
\subsection{Permutation groups} \label{permutazioni}
We recall some definitions and results that we will use in the following. 

A group $G$ acting on a set $\ \Omega=\{1,\ldots,d\}$ is $k$-transitive, with $k \leq d$, if, given two ordered $k$-tuples $(m_1,\ldots,m_k)$ and $(t_1,\ldots,t_k)$ of distinct points in $\Omega$, there is an element $g \in G$ such that sends $g \cdot m_i=t_i$ for every $i=1,\ldots,$. If $k=1$ we say that $G$ is transitive. 

We state some results on transitive permutation groups that we will use in the following. See for instance \cite[Chapter 8]{Isaacs} for a more complete treatment.

\begin{lem}\label{lemma1}
Let $G$ be a group acting transitively on $\Omega$, let $i \in \Omega$ and $k \leq d-1$. 
The group $G$ is $k$-transitive on $\Omega$ if and only if the stabilizer of $i$ in $G$ is $(k-1)$-transitive on $\Omega \setminus \{i\}$. 
\end{lem}

We recall that a {block} is a non-empty subset $B \subset \Omega$ such that either $g\cdot B=B$ or $(g\cdot B) \cap B = \emptyset$ for all $g \in G$. We say that $G$ is {imprimitive} if its action preserves non-trivial blocks and {primitive} otherwise. A 2-transitive permutation group is primitive, but the converse is not always true. 
\begin{lem}\label{block}
Let $G$ be a group acting transitively on $\Omega$ and let $B$ be a block. Then $|B|$ divides $d=|\Omega|$ and in $\Omega$ there are exactly ${|\Omega|}/{|B|}$ disjoint blocks, all with the same cardinality.
\end{lem}

\begin{lem}\label{lemma2}
Let $G$ be a primitive group on $\Omega$, let $A \subset \Omega$ such that $0 < |A| \leq d-2$ and the stabilizer of $A$ is transitive on $\Omega \setminus A$.  
Then $G$ is $2$-transitive on $\Omega$. 
\end{lem}
We recall that the monodromy group of $\pi_P$ is imprimitive if and only if the projection is decomposable (\cite[Remark 2.2]{PS}). 

\section{Topology of finite morphisms}\label{sec3}
We introduce the following definition of what will be for us a \emph{fibration}. For a more complete treatment see \cite[Chapter III sect. 8]{BHPVdV}.
\begin{defn}
A \emph{fibration} is a proper surjective morphism $f: X \to Y$ with connected fibres from a smooth complex variety to a smooth quasi-projective curve. 

Let $y \in Y$ be a point. A fibre of $f$ is $F:=f^*(y)=\sum n_iY_i$, where the $Y_i$'s are irreducible components and $n_i \geq 1$ are their multiplicities. A fibre $F$ is called \emph{multiple} if $\gcd\{n_i\}:=m>1$; we will write $F=mE$, where $E=\sum t_iY_i$ with $\gcd\{t_i\}=1$. 
\end{defn}
We remark that $f$ is flat since $Y$ is smooth. 
We will use the following classical result on homotopy of fibrations. 
\begin{prop}\label{lem:Noriprelim}
Let $f: X \to Y$ be a fibration. If $f$ does not have multiple fibres, then the following sequence is exact
 \begin{equation*}
     \pi_1(F) \to \pi_1(X)\to \pi_1(Y)\to 1      
    \end{equation*}
where $F$ is a general fibre of $f$. 
\end{prop}
The proof is based on a combination of the techniques in \cite[Lemma 1.5]{Nori} that proved the result in the case where every fibre has at least a reduced component, and in \cite{Serrano} that proved the exactness of the sequence for homology groups.

\subsection{Monodromy group of families of maps} \label{sec:limit}
We refer the reader to \cite[ChapterIII.9]{Hartshorne} and \cite[Chapter4.6.7]{Sernesi} for background material about families of algebraic space.

Let $\cX \to \bP^1$ be a (flat) family of projective varieties of dimension $m$ in $\bP^N$ parametrized by $\bP^1$ and let $Y$ be a smooth projective variety. Consider the following diagram
\begin{equation*}
    \xymatrix{
    \cX \ar[r]^F \ar[d]^p & Y \times \bP^1 \ar[dl]^q \\
    \bP^1 & }
\end{equation*}
Let $p$ and $q$ be proper surjective maps with connected fibres. 

A general fibre $ X_s$ of $p$, with $s \in \bP^1$, is a smooth projective variety of dimension $m$ in $\bP^N$, together with a finite dominant morphism $f_s: X_s \to Y$ of degree $d$. 
Denote by $B_s \subset Y$ the branch divisor of $f_s$ and $R_s \subset X_s$ its ramification divisor. Let $M(f_s)$ be the monodromy group of $f_s$.

Let $\mathcal{R}\subset \cX$ be the ramification divisor of $F$, i.e. $\cX \setminus \mathcal{R}=\lbrace  X_s \setminus R_s\  |\ s \in \bP^1 \rbrace$. Let moreover $\mathcal{B}$ be the branch divisor of $F$, i.e. $(Y \times \bP^1) \setminus \mathcal{B}=:\cV=\lbrace  (Y \setminus B_s) \times \{s\}\  |\ s \in \bP^1 \rbrace$, open inside $Y \times \bP^1$. Let $M(F)$ be the monodromy group of $F$.

If all the varieties $X_s$ are reduced we can deduce the following property of monodromy groups.
\begin{prop}\label{lemmagenerale}
In the above setting, assume that every fibre of $p$ is reduced; then 
$$M(F)\cong M(f_s)$$ for a general $s \in \bP^1$. 
\end{prop}
\begin{proof}
By assumptions, there is no $s \in \bP^1$ such that $(Y,s) \subset \mathcal{B}$. Therefore, the map $q': (Y \times \bP^1)\setminus \mathcal{B} \to \bP^1$ is a fibration with reduced and connected fibres. Thanks to Proposition \ref{lem:Noriprelim}, the following sequence is exact for a general $s \in \bP^1$
\begin{equation*}
\xymatrix{
    \pi_1(Y\setminus B_s) \ar@{->>}[r] & \pi_1((Y \times \bP^1) \setminus \mathcal{B}) \ar[r] & \pi_1(\bP^1)=1. }
\end{equation*}

Combining this together with the monodromy map $\mu$, we have 
\begin{equation*}
    \xymatrix{
   \pi_1(Y \setminus B_s) \ar@{->>}[r] \ar@{->>}[d]^\mu & \pi_1((Y \times \bP^1) \setminus \mathcal{B})  \ar@{->>}[d]^\mu\\
   M(f_s) \ar@{->>}[r] & M(F)}
\end{equation*}
Moreover, if we have a subvariety $Z \subset Y\times \bP^1$ that is not contained in $\cB$, then 
\begin{equation*}\label{iniettiva}
    \pi_1(Z) \hookrightarrow \pi_1((Y \times \bP^1) \setminus \mathcal{B}).
\end{equation*}
Taking $Z$ as a general fibre of $q'$, then we have also that the map between the monodromy groups is injective. Hence $M(F) \cong M(f_s)$. 
\end{proof}

\begin{cor}
In the above assumptions, let $X_0$ be a fibre of $p$ and $f_0: X_0 \to Y$ its dominant morphism. Then $M(f_0) \subseteq M(f_s)$.
\end{cor}

More generally, assume that the fibration $p:\cX \to \bP^1$ has a singular fibre $F_0=\sum n_i Z_i$ with at least a reduced component $Z_i$. 
Let $g_0$ be the map $f_0$ restricted to $Z_i$, i.e. $g_0 = (f_0)_{|Z_i}:Z_i \to Y$, dominant morphism of degree strictly lower than $d=\deg(f_0)$. 

\begin{prop}\label{redcomponent}
For a general $s$ in a neighbourhood of $\ 0$,
$$M(g_0) \subset M(f_s).$$
\end{prop}
\begin{proof}
Let $Z:=Z_i$ and, by abuse of notation, we will still denote by $B_0$ the branch divisor of $g_0$ and $R_0$ its ramification divisor. Let $\sigma \in M(g_0)$. Then there exists $[\gamma] \in \pi_1(Y \setminus B_0,y)$ such that $\mu (\gamma)=\sigma$. Let $\gamma$ be a representative of $[\gamma]$ and let $\Tilde{\gamma}$ be its lifting to $\Tilde{Z}:=Z\setminus R_0$. 
The path $\Tilde{\gamma}$ is the image of $[0,1]$ inside $\Tilde{Z}$, such that $\Tilde{\gamma}(0)=z_0$ and $\Tilde{\gamma}(1)=z_1$, where $z_0,z_1$ are two distinct points in the fibre $g_0^{-1}(y)$. We can also assume that we avoid the points in which $Z$ meets the other components $Z_j$ of $F_0$.
The path is compact and consider a tubular neighbourhood $U$ of it. Then, by assumptions, the fibration $p$ restricted to $U$ is a locally trivial fibration by Ehresmann's Theorem (\cite[Lemma 4.2]{Catanese}, \cite[Sec 4]{Massey} for manifolds with boundary). Hence, the path $\Tilde{\gamma}$ can be moved in $U$ to a path $\Tilde{\gamma}_s$ in a fibre $X_s,\ s \neq 0$. 
Therefore, $M(f_s)\owns\mu\left(p(\Tilde{\gamma}_s)\right)= \sigma$.  
\end{proof}
\section{Projections of general hypersurfaces}\label{sec:main}
We want now apply the previous construction to the following situation. Let $\cX \to \Delta$ be a pencil of hypersurfaces in $\bP^{n+1}$ parametrised by a disc $\Delta$, small neighbourhood of $0$. 
Its general element is a general hypersurface $X$ and the hypersurface $X_0$ is given by a general hypersurface $Y$ of degree $d-1$ and a hyperplane $H$. 
Let $P \in \bP^{n+1}$ be a point and $\bP^n$ an hyperplane not containing $P$ and consider  
\begin{equation*}
    \xymatrix{
    \widetilde{\bP^{n+1}} \ar[d]^\nu \ar[dr]^{\widetilde{\pi_P}} & \\
    \bP^{n+1} \ar@{-->}[r]^{\pi_P} & \bP^n
    }
\end{equation*}
where $\nu$ is the blow up of the projective space at $P$ and $\pi_P$ is the projection of $\bP^{n+1}$ form $P$. 
Consider the linear projection $\pi_s:= (\pi_P)_{|X_s}: X_s \dashrightarrow \bP^n$ of a general element $X_s$ in $\cX$ with $s \in \Delta$. 
Degenerating the hypersurface $X$ to $X_0$ as $t$ goes to $0$, the point $P$ degenerate onto a point $P_0 \in \bP^{n+1}$. Note that, if the point $P$ is in $X$, the point $P_0$ is in $X_0$. After a change of coordinates, we can think the point $P$ as fixed. We have the following diagram
\begin{equation*}
    \xymatrix{
    \widetilde{\cX} \ar[r]^{\widetilde{\pi_P}} \ar[d]^p & \Delta \times \bP^n \ar[dl]^q \\
    \Delta & }
\end{equation*}
where $\widetilde{\cX}$ is the family of the strict transforms $\widetilde{X_s} \subset \widetilde{\bP^{n+1}}$ for every $X_s \subset \cX$ and $\widetilde{\pi_s}: \widetilde{X_s} \to \bP^n $ is a dominant morphism of degree $d$. We recall that, if $P \notin X_s$, then $\widetilde{\pi_s}=\pi_s$ and moreover, the monodromy group does not change when we blow up a smooth point of $X_s$. 

\begin{lem}\label{induction}
Let $P \in \bP^{n+1}$ and let $\pi_0$ be the map $\pi_P$ restricted to $Y$. Then, the monodromy group $M(\pi_0)$ is contained in the monodromy group $M(\pi_s)$ for a general $s \in \Delta$.
\end{lem}
\begin{proof}
If in the limit $P \notin Y \cap H$, the singular locus of $X_0$, then all the varieties in $\widetilde{\cX}$ are reduced. We can apply Proposition \ref{lemmagenerale} and have that $M(\widetilde{\pi_s}) \cong M(\widetilde{\pi_P})$. Moreover, we have that $$M(\pi_0) \subset M(\widetilde{\pi}:\widetilde{X_0} \to \bP^n) \subset  M(\widetilde{\pi_s})=M(\pi_s).$$

If $P \in Y \cap H$, then we have that $\widetilde{Y}$ is a reduced component of $\widetilde{X_0}$ and $P$ is a smooth point of $Y$. By Proposition \ref{redcomponent}, the monodromy group of $\pi_0$ is still contained in the monodromy group of a general fibre $M(\pi_s)$.
\end{proof}

\medskip

\subsection{Monodromy of general hypersurfaces}
Let $X \subset \bP^{n+1}$ be a general hypersurface of degree $d >1$ and let $P \in \bP^{n+1}$ be a point. Let $\pi_P$ be the linear projection of $X$ from $P$ and let $M(\pi_P)$ be its corresponding monodromy group. We recall that if $P \in X$, then $M(\pi_P) \subseteq S_{d-1}$, while if $P \notin X$, then $M(\pi_P)\subseteq S_d$. 

\begin{rmk}\label{d2}
Every point $P$ is uniform if $d=2$. Indeed, there are no proper transitive subgroup of $S_2$.
\end{rmk}

As a consequence of Theorem \ref{autom} we get the following.
\begin{thm}\label{d3}
Let $d=3$. Then every point is uniform.
\end{thm}
\begin{proof}
If $P \in X$, the degree of $\pi_P: X \dashrightarrow \bP^n$ is two and so $M(\pi_P)=S_2$. 

Let now $P \notin X$ and assume by contradiction that $P$ is non uniform. Then $M(\pi_P)=A_3$ and so $X$ has a non trivial automorphism. This is a contradiction of Theorem \ref{autom}. Hence $M(\pi_P)=S_3$ for every $P \notin X$. 
\end{proof}

We are now ready to prove the main result of the paper.
\begin{thm}\label{interni}
Let $X \subset \bP^{n+1}$ be a general hypersurface of degree $d>1$. Then all the points $P \in \bP^{n+1}$ are uniform.
\end{thm}
\begin{proof}
The result has been already proven for $d \leq 3$ (Theorem \ref{d3}). 

We work by induction on $d=\deg(X)$. Assume that every point is uniform for a general hypersurface of degree $d-1$. Let $P \in \bP^{n+1}$ be a point and degenerate $X$ onto $X_0=Y \cup H$ as in Lemma \ref{induction}. Recall that the hypersurface $Y$ is general of degree $d-1$.
\medskip

Assume that $P \in X$, hence $P \in X_0$ by degeneration. If $P \in Y$, by induction $M((\pi_0)_{|Y})=S_{d-2}$. Moreover, $S_{d-2} \subseteq M(\pi_P)$ by Lemma \ref{induction}. Therefore, $M(\pi_P)$ is a transitive group acting on a general fibre of $\pi_P$ and, by construction, it contains a subgroup that is $d-2$ transitive on $d-2$ points of the fibre. Therefore $M(\pi_P)$ is $d-1$ transitive by Lemma \ref{lemma1}, and so $P$ is uniform. If $P \in H$, then $M((\pi_0)_{|Y})=S_{d-1}$. Therefore, applying Lemma \ref{induction} we conclude that $P$ is uniform for $X$.

\medskip
Assume now that $P \notin X$. We recall that, in the degeneration, the point $P$ may be in $X_0$. If $P \notin Y$, then $M((\pi_0)_{|Y})=S_{d-1}$. Moreover, by Lemma \ref{induction}, it is contained in $M(\pi_P)$. Hence it is a group acting transitively on a general fibre of $\pi_P$ and that contains a subgroup that is $d-1$ transitive on $d-1$ points of the fibre. Therefore, by Lemma \ref{lemma1}, $M(\pi_P)$ is $d$ transitive, i.e. the point $P$ is uniform.

If $P \in Y $, then $S_{d-2}= M((\pi_0)_{|Y})$. By Lemma \ref{induction} we have that $M(\pi_P)$ contains a subgroup that is $d-2$ transitive on $d-2$ points of a general fibre. If moreover $M(\pi_P)$ is primitive, then by Lemma \ref{lemma2} we have that it is $2$-transitive. If we apply again Lemma \ref{lemma1} we get that it is $d$-transitive on a general fibre, i.e. $P$ is uniform. 

We are then left to prove that the action of $M(\pi_P)$ is primitive. If $d \geq 5$ the action of $M(\pi_P)$ is clearly primitive since $d-2$ does not divide $d$ (see Lemma \ref{block}). If $d=4$, assume by contradiction that the map $\pi_P$ is decomposable. The only possibility is that it factors via two maps of degree two $X \stackrel{2:1}{\to} Y \stackrel{2:1}{\to} \bP^n.$ The first map can be seen as an involution of the general quartic, hence a non trivial automorphism of $X$. This contradicts Theorem \ref{autom}. 

Therefore, every point is uniform. 

\end{proof}

\section*{Acknowledgements}
The author is supported by MIUR: Dipartimenti di Eccellenza Program (2018-2022) - Dept. of Math. Univ. of Pavia and by PRIN 2017 "Moduli spaces and Lie Theory" code 2017YRA3LK\_003. 
I would like to thank Gian Pietro Pirola for introducing me to the problem and for all the help he gave during the preparation of this paper. I also thank Ciro Ciliberto, Riccardo Moschetti, Lidia Stoppino and Thomas Dedieu for helpful discussions and suggestions.

\bibliographystyle{alpha}
\bibliography{bib}
 
\end{document}